\documentclass[twoside]{article}

\usepackage{amsmath,amssymb,amsfonts,amsthm,fancyhdr}          % 使用AMS数学公式符号字体定理,页眉,扩展的color
\usepackage{epsfig,graphicx,picins,picinpar,subfigure}         % 使用图形
\usepackage{pstricks}                                          % 作图时使用（在此用于文章编号加框）
\usepackage{fancyvrb}                                          % 支持抄录 --- 仅用于生成源代码
\usepackage[numbers,sort&compress]{natbib}                     % 使参考文献连续编号的引用以压缩方式出现
\usepackage{graphicx,amssymb,mathrsfs,amsmath,latexsym,amsfonts,amsthm,bbding}

\catcode`@=11 \long\def\@makefntext#1{\noindent #1}
\newskip\tabcentering \tabcentering=1000pt plus 1000pt minus 1000pt

\def\MCH#1#2{\setbox0=\hbox{\raise#1\hbox{#2}}\smash{\box0}}% move char

\let\@oddfoot\@empty  \let\@evenfoot\@empty

\def\@cite#1#2{{#1\if@tempswa,#2\fi}}

\renewcommand\@biblabel[1]{[${#1}$]}%将{}中的#1改变，可得到所需形式

\def\@evenhead{\thepage \hbox to\textwidth{\hfill\footnotesize{\it Guo TieXin, Zeng XiaoLin~~~}}}
\def\@oddhead{\hbox to \textwidth{\footnotesize{\it
An $L^{0}({\cal F},R)-$valued function's intermediate value theorem and its applications} \hfill\thepage}}

\begin{document}

%====================================================================%
%  带节号的编号方式： 对应下面的 Method 2
%====================================================================%
%%\renewcommand{\thetable}{\arabic{section}.\arabic{table}}
\renewcommand{\theequation}{\arabic{section}.\arabic{equation}}
\makeatletter      % '@' is now a normail "letter" for TeX
\@addtoreset{equation}{section}
\makeatother       % '@' is restored as a "non-letter" character for TeX

%%\renewcommand{\thefigure}{\arabic{section}.\arabic{figure}}
%====================================================================%
%  不带节号的编号方式： 对应下面的 Method 1
%====================================================================%
\renewcommand{\thetable}{\arabic{table}}
\renewcommand{\thefigure}{\arabic{figure}}

%====================================================================%
%          中文文档定理结构的设置,重定义一些正文相关标题             %
%                    针对英文稿设置                                  %
%====================================================================%
%\usepackage{amsthm}
\newtheoremstyle{mythm}{3pt}{3pt}{}{}{\bfseries}{}{2mm}{}
\theoremstyle{mythm}

\newtheorem{definition}{{\bf{Definition}}}[section]
\newtheorem{proposition}{{\bf{Proposition}}}[section]
\newtheorem{property}{{\bf{Property}}}[section]
\newtheorem{theorem}{{\bf{Theorem}}}[section]
\newtheorem{lemma}{{\bf{Lemma}}}[section]
\newtheorem{corollary}{{\bf{Corollary}}}[section]
\newtheorem{axiom}{\hspace{2em}{\bf{Axiom}}}[section]
\newtheorem{exercise}{\hspace{2em}{\bf{Exercise}}}[section]
\newtheorem{question}{\hspace{2em}{\bf{Question}}}
\newtheorem{example}{{\bf{Example}}}
\newtheorem{notation}{\hspace{2em}{\bf{Notation}}}
\newtheorem{remark}{\hspace{2em}{\bf{Remark}}}

%%%%%%%%%%%%%%%%%%%%%%%%%%%%%%%%%%%%%%%%%%%%%%%%%%%%%%%%%%%%%%%%
%%-------------------- 作者提供的信息 ------------------------%%
%%%%%%%%%%%%%%%%%%%%%%%%%%%%%%%%%%%%%%%%%%%%%%%%%%%%%%%%%%%%%%%%
          % 仅用于页眉, 与\title{}中的一致 -- 用中文
\newcommand{\authorsinfo}{Corresponding author}

%%%%%%%%%%%%%%%%%%%%%%%%%%%%%%%%%%%%%%%%%%%%%%%%%%%%%%%%%%%%%%%%
% 标题,作者, 通信地址定义
%%%%%%%%%%%%%%%%%%%%%%%%%%%%%%%%%%%%%%%%%%%%%%%%%%%%%%%%%%%%%%%%
%%\begin{CJK}{GBK}{song}
\title{\bf An $L^{0}({\cal F},R)-$valued function's intermediate value theorem and its applications to random uniform convexity\thanks{This work is supported by the National Natural Science Foundation of China (No. 10871016).}}
%%%%%%%%%%%%%%%%%%%%%%%%%%%%%%%%%%%%%%%%%%%%%%%%%%%%%%%%%%%%%%%%
% 作者姓名与单位 ： 三种形式中选一种
% 后面中文摘要中的名字和单位同样处理
% ---------------------
% 第一种形式: 单一作者
% ---------------------
%\author{\textsc{Author1}\\[-1pt]
%(\textit{\zihao{-5} Author's Working Unit (Up to Department), Province Zip ~ Code}) \\[-2pt]}

% ---------------------
% 第二种形式: 同一单位 多个作者 -- 名字左右并列,
% ---------------------
\author{\bf{Guo TieXin$^{1}$\footnote{\authorsinfo}\,,\,\,Zeng XiaoLin}$^{2}$\\[-1pt]
\textit{\small LMIB and School of Mathematics and Systems Science,
Beihang University,
Beijing 100191, PR.} \\[-2mm]
\textit{\small China}\\[-2pt]
\textit{\small E-mail: $^{1}$txguo@buaa.edu.cn $^{2}$xlinzeng@ss.buaa.edu.cn}}

% ---------------------
% 第三种形式: 不同单位 多个作者 -- 名字与单位上下并列
% ---------------------
%\author{\textsc{Author1}\\[-1pt]
%(\textit{\zihao{-5} First Author's Working Unit (Up to Department), Province Zip ~ Code}) \\[-2pt]
%\textsc{Author2}\\[-1pt]
%(\textit{\zihao{-5} Second Author's Working Unit (Up to Department), Province Zip ~ Code}) \\[-2pt]}
%%%%%%%%%%%%%%%%%%%%%%%%%%%%%%%%%%%%%%%%%%%%%%%%%%%%%%%%%%%%%%%%
\date{}
\maketitle

\vspace{-6mm}
%%%%%%%%%%%%%%%%%%%%%%%%%%%%%%%%%%%%%%%%%%%%%%%%%%%%%%%%%%%%%%%%
%  英文文摘要
%%%%%%%%%%%%%%%%%%%%%%%%%%%%%%%%%%%%%%%%%%%%%%%%%%%%%%%%%%%%%%%%

\vskip 1 mm

\noindent{\small {\small\bf Abstract}~~Let $(\Omega,{\cal F},P)$ be a probability space and $L^{0}({\cal F},R)$ the algebra of
equivalence classes of real-valued random variables on $(\Omega,{\cal F},P)$. When $L^{0}({\cal F},R)$ is endowed
with the topology of convergence in probability, we prove an intermediate value theorem for a continuous local
function from $L^{0}({\cal F},R)$ to $L^{0}({\cal F},R)$. As applications of this theorem, we first give several
useful expressions for modulus of random convexity, then we prove that a complete random normed module
$(S,\|\cdot\|)$ is random uniformly convex iff $L^{p}(S)$ is uniformly convex for each fixed positive number $p$
such that $1<p<+\infty$. \ \

\vspace{1mm}\baselineskip 12pt


\noindent{\small\bf MR(2000) Subject Classification}~ 46A22, 46B20, 46E30\ \ {\rm }}
%%%%%%%%%%%%%%%%%%%%%%%%%%%%%%%%%%%%%%%%%%%%%%%%%%%%%%%%%%%%%%%%
%  正文由此开始
%%%%%%%%%%%%%%%%%%%%%%%%%%%%%%%%%%%%%%%%%%%%%%%%%%%%%%%%%%%%%%%%

\section{Introduction}

Based on the analysis of stratification structure on random normed modules, in \cite{Guo-Zeng} we introduced the notions of random strict convexity and random uniform convexity in random normed modules and gave the perfect relation between random strict convexity and classical strict convexity. However, when we also attempted to give the similar relation between random uniform convexity and classical uniform convexity in \cite{Guo-Zeng} we encountered some difficulties, which made us only obtain a not very pleasant result in \cite{Guo-Zeng}. The purpose of this paper is to overcome the difficulties so that we can give the perfect relation between random uniform convexity and classical uniform convexity. Besides, this paper also gives several useful expressions for modulus of random convexity. In particular, we give an $L^{0}({\cal F},R)-$valued function's intermediate value theorem, which will play an essential role in the proofs of the above main results.

To introduce the main results of this paper, let us first recall some notation and terminology together with some known notions.

Throughout this paper, $(\Omega,{\cal F},P)$ always denotes a probability space, $K$ the scalar field $R$ of real numbers or $C$ of complex numbers, $\bar{L}^{0}({\cal F},R)$ the set of equivalence classes of extended real-valued random variables on $(\Omega,{\cal F},P)$ and ${L}^{0}({\cal F},K)$ the algebra of equivalence classes of $K$-valued random variables on $(\Omega,{\cal F},P)$.

It is well known from \cite{Dunford-Schwartz} that $\bar{L}^{0}({\cal
F},{R})$ is a complete lattice under the ordering $\leqslant$:
$\xi\leqslant \eta$ iff $\xi^{0}(\omega)\leqslant\eta^{0}(\omega)$
for $P-$almost all $\omega$ in $\Omega$ (briefly, a.s.), where
$\xi^{0}$ and $\eta^{0}$ are arbitrarily chosen representatives of
$\xi$ and $\eta$, respectively. Furthermore, every subset $A$ of
$\bar{L}^{0}({\cal F},{R})$ has a supremum, denoted by $\vee A$, and
an infimum, denoted by $\wedge A$, and there exist two sequences
$\{a_{n},n\in N\}$ and ${\{b_{n},n\in N}\}$ in $A$ such that
$\vee_{n\geqslant1}$ $a_{n}=\vee A$ and $\wedge_{n\geqslant1}$
$b_{n}=\wedge A$. If, in addition, $A$ is directed (accordingly,
dually directed), then the above $\{a_{n},n\in N\}$
(accordingly, $\{b_{n},n\in N\}$) can be chosen as nondecreasing
(accordingly, nonincreasing). Finally $L^{0}({\cal F},R)$, as a
sublattice of $\bar{L}^{0}({\cal F},{R})$, is complete in the sense
that every subset with an upper bound has a supremum (equivalently,
every subset with a lower bound has an infimum).

 Specially, let
$\bar{L}^{0}_{+}=\{\xi\in \bar{L}^{0}({\cal
F},{R})\,|\,\xi\geqslant0\}, L^{0}_{+}=\{\xi\in L^{0}({\cal
F},R)\,|\,\xi\geqslant0\}$ and $L^{0}_{++}=\{\xi\in L^{0}({\cal
F},R)\,|\,\xi>0$ on $\Omega\}$, where $\xi>0$ on $\Omega$ means that $\xi^{0}(\omega)>0$ a.s. on $\Omega$ for an arbitrarily chosen representative $\xi^{0}$ of $\xi$.

 \begin{definition}[Guo \cite{Guotx-basictheory}]\label{def:RNmodule} An ordered pair $(S,\|\cdot\|)$ is called a random normed space (briefly, an $RN$ space) over $K$ with base $(\Omega,{\cal F},P)$ if $S$ is a linear space over $K$ and $\|\cdot\|$ is a mapping from $S$ to $L^{0}_{+}$ such that the following axioms are satisfied:\\
\indent($RN$-1) $\|\alpha x\|=|\alpha|\|x\|,\forall\alpha\in K$ and $x\in S$;\\
\indent($RN$-2) $\|x+y\|\leq\|x\|+\|y\|,\forall x,y\in S$;\\
\indent($RN$-3) $\|x\|=0$ implies $x=\theta$ (the null vector in $S$).\\
Where $\|x\|$ is called the random norm of the vector $x$.

In addition, if $S$ is a left module over the algebra ${L}^{0}({\cal F},K)$ and $\|\cdot\|$ also satisfies the following:\\
\indent($RNM$-1) $\|\xi x\|=|\xi|\|x\|,\forall\xi\in {L}^{0}({\cal F},K)$ and $x\in S$.\\
Then such an $RN$ space $(S,\|\cdot\|)$ is called a random normed module (briefly, an $RN$ module) over $K$ with base $(\Omega,{\cal F},P)$, such a random norm $\|\cdot\|$ is called an ${L}^{0}-$norm.
\end{definition}

\begin{example}\label{exa:RNmodule}
($L^{0}({\cal F},K),|\cdot|$) is an $RN$ module over $K$ with base $(\Omega,{\cal F},P)$, where the $L^{0}-$norm $|x|$ of any $x\in L^{0}({\cal F},K)$ is defined to be the equivalence class of the composite function $|x^{0}|:\Omega\rightarrow [0,+\infty)$, namely $|x^{0}|(\omega)=|x^{0}(\omega)|,\forall\omega\in\Omega$, where $x^{0}$ is an arbitrarily chosen representative of $x$.
\end{example}

 \begin{definition}[Guo \cite{Guotx-basictheory}]\label{def:weaktop}
 Let $(S,\|\cdot\|)$ be an $RN$ space over $K$ with
base $(\Omega,{\cal F},P)$. Given any $\epsilon>0,0<\lambda<1$, let
$N(\epsilon,\lambda)=\{x\in S~|~P\{\omega\in\Omega:\|x\|(\omega)<\epsilon\}>1-\lambda\}$, then
the family ${\cal U}_{\theta}=\{N(\epsilon,\lambda)~|~
\epsilon>0,0<\lambda<1\}$ forms a local base at the null element $\theta$ of some metrizable
linear topology for $S$, called {\it the
$(\epsilon,\lambda)-$topology for $S$}.
\end{definition}

\begin{proposition}[Guo \cite{Guotx-basictheory}]\label{prop:weaktop}
 Let $(S,\|\cdot\|)$ be an $RN$ module over $K$ with base $(\Omega,{\cal F},P)$, then\\
\indent(1) The $(\epsilon,\lambda)-$topology for $L^{0}({\cal F},K)$ is exactly the topology of convergence in probability $P$;\\
\indent(2) $L^{0}({\cal F},K)$ is a topological algebra under the $(\epsilon,\lambda)-$topology.\\
\indent(3) $S$ is a topological module over the topological algebra $L^{0}({\cal F},K)$ when $S$ and $L^{0}({\cal F},K)$ are endowed with their respective $(\epsilon,\lambda)-$topologies.
\end{proposition}

 \begin{definition} [Filipovi\'c, Kupper and Vogelpoth \cite{Filipovic}]
Let $(S,\|\cdot\|)$ be an $RN$ space over $K$ with base $(\Omega,{\cal F},P)$. Let $N_{\theta}(\epsilon)=\{x\in S~|~\|x\|\leq\epsilon\}$ for any $\epsilon\in L^{0}_{++}$, and let ${\cal T}_c$=$\{G\subset S~|~$for each $y\in G$ there exists some $\epsilon\in L^{0}_{++}$ such that $y+N_{\theta}(\epsilon)\subset G$ $\}$. Then ${\cal T}_c$ is a Hausdorff topology for $S$, called the locally $L^{0}-$convex topology.
\end{definition}

 Under the locally $L^{0}-$convex topology $L^{0}({\cal F},K)$ is a topological ring, which means that the locally $L^{0}-$convex topology for $L^{0}({\cal F},K)$ is not necessarily a linear topology, see \cite{Filipovic} for details. When $(S,\|\cdot\|)$ is an $RN$ module and is endowed with the locally $L^{0}-$convex topology, it is a Hausdorff topological module over the topological ring $L^{0}({\cal F},K)$.

In \cite{KV}, to study the subdifferential of a conditional convex risk measure, Kupper and Vogelpoth proved an interesting intermediate value theorem for a continuous local function $f$ from $L^{0}({\cal F},R)$ to $L^{0}({\cal F},R)$ when $L^{0}({\cal F},R)$ is endowed with its locally $L^{0}-$convex topology, where $f$ is called \emph{local} if $\tilde{I}_Af(x)=\tilde{I}_Af(\tilde{I}_Ax)$ for any $A\in{\cal F}$ and $x\in L^{0}({\cal F},R)$, here $\tilde{I}_A$ denotes the equivalence class of the characteristic function $I_A$ of $A$. However, in this paper we need an intermediate value theorem when $L^{0}({\cal F},R)$ is endowed with its $(\epsilon,\lambda)-$topology. Precisely speaking, we prove Theorem \ref{thm:IVT} below. For the sake of conciseness we briefly denote by $[\alpha_1,\alpha_2]$ the set $\{\alpha\in L^{0}({\cal F},R)~|~\alpha_1\leq\alpha\leq\alpha_2\}$ for any $\alpha_1,\alpha_2\in L^{0}({\cal F},R)$ with $\alpha_1\leq\alpha_2$.

\begin{theorem}\label{thm:IVT}
Let $f:L^{0}({\cal F},R)\rightarrow L^{0}({\cal F},R)$ be a continuous local function when $L^{0}({\cal F},R)$ is endowed with its $(\epsilon,\lambda)-$topology and $Y_1,Y_2\in L^{0}({\cal F},R)$ such that $Y_1\leq Y_2$. Then for any $\xi\in [f(Y_1)\wedge f(Y_2),f(Y_1)\vee f(Y_2)]$ there exists $\eta\in [Y_1,Y_2]$ such that $f(\eta)=\xi$.
\end{theorem}

 In the sequel of this paper all the $RN$ modules are always assumed to be endowed with  the $(\epsilon,\lambda)-$topology.

Given an element $\xi$ in $\bar{L}({\cal F},R)$ with a representative $\xi^{0}$, we use $[\xi>0]$ for the equivalence class of the set $\{\omega\in\Omega~|~\xi^{0}(\omega)>0\}$. Let $(S,\|\cdot\|)$ be an $RN$ space over $K$ with
base $(\Omega,{\cal F},P)$ and put $\xi=\vee\{\|x\|~|~x\in S\}$, then a representative of $H(S):=[\xi>0]$ is called a support of $S$.

Let $\widetilde{\cal F}$ denote the set of equivalence classes of elements in ${\cal F}$ and $A,B\in\widetilde{\cal F}$ with respective representatives $A_0$ and $B_0$, $A\subset B$ means $P(A_0\backslash B_0)=0$, $A\cup B$ and $A\cap B$ denote the equivalence classes of $A_0\cup B_0$ and $A_0\cap B_0$, respectively.  For simplifying notation we also use $I_A$ for the equivalence class of ${I}_{A_{0}}$.

First of all, let us recall the notion of modulus of random convexity as follows:

Let $(S,\|\cdot\|)$ be a complete $RN$ module over $K$ with base $(\Omega,{\cal F},P)$ such that $P(H(S))>0$. Define
\begin{equation}
\begin{split}
\delta_D(\epsilon)
=\bigwedge\left\{{I}_D-{I}_D\left\|\frac{x+y}{2}\right\|:x,y\in S(1)\textmd{ with } B_{xy}\supset D\textmd{ and
}{I}_D\|x-y\|\geq\epsilon{I}_D\right\}
\end{split} \tag{1.1}
\end{equation}
for any $D\in\widetilde{\cal F}$ with $D\subset H(S)$ and $P(D)>0$ and $\epsilon\in L^{0}_{+}$ such that
$0<\epsilon\leq 2$ on $D$, where\\
\indent $S(1)=\{x\in S:\|x\|=I_A$ for some $A\in\widetilde{\cal F}$ with $P(A)>0\}$,\\
\indent $A_x=[\|x\|>0], A_{xy}=A_{x}\cap A_{y}$ and $B_{xy}=A_{xy}\cap A_{x-y},\forall x,y\in S$.\\
It is known that $\delta_D(\epsilon)I_G=\delta_G(\epsilon)$ for any $G\in\widetilde{\cal F}$ with $G\subset D$. Then the mapping $\delta_{H(S)}(\cdot):\bar{\mathscr{E}}(S)\rightarrow L^{0}_{+}$ defined by (1.1) is called \emph{the modulus of random convexity} of $S$, where $\bar{\mathscr{E}}(S)=\{\epsilon\in L^{0}_{+}:0<\epsilon\leq 2
\textmd{ on }H(S)\}$.

When $(\Omega,{\cal F},P)$ is a trivial probability space, namely, ${\cal F}=\{\Omega,\emptyset\}$ and $P(\Omega)=1$, then $S$ becomes an ordinary Banach space and (1.1) degenerates to
$$\delta(\epsilon)=\inf\left\{1-\left\|\frac{x+y}{2}\right\|:x,y\in X,\|x\|=\|y\|=1\textmd{ and }\|x-y\|\geq\epsilon\right\}$$ for any $\epsilon\in R$ with $0<\epsilon\leq 2$, which is just the classical modulus of convexity of the Banach space $S$. It is well known from \cite{Mgginson} or \cite{Linden} that $\delta(\epsilon)$ has the following two useful expressions when $K=R$ and dim($S$)$\geq 2$:
\begin{equation}
\begin{split}
\delta(\epsilon)=&\inf\left\{1-\left\|\frac{x+y}{2}\right\|:x,y\in X,\|x\|=\|y\|=1\textmd{ and }\|x-y\|=\epsilon\right\}\\
=&\inf\left\{1-\left\|\frac{x+y}{2}\right\|:x,y\in X,\|x\|\leq 1,\|y\|\leq 1\textmd{ and }\|x-y\|\geq\epsilon\right\}.
\end{split} \tag{1.2}
\end{equation}

One may naturally ask if there exist such expressions for modulus of random convexity when the base space $(\Omega,{\cal F},P)$ of $S$ is not trivial. This problem involves a detailed discussion of the notions of $L^{0}-$independence and quasi-rank in real $RN$ modules, from which we know that for every complete $RN$ module $S$ there exists a unique $G(S)$ in $\widetilde{\cal F}$ with $G(S)\subset H(S)$ such that $\delta_{H(S)\backslash G(S)}(\epsilon)=I_{H(S)\backslash G(S)}$ for any $\epsilon\in\bar{\mathscr{E}}(S)$, and such that the quasi-rank of $S$ on $G(S)$ is not less than 2 when $P(G(S))>0$. Following is our second main result:

\begin{theorem}\label{thm:expr}
 Let $(S,\|\cdot\|)$ be a complete real $RN$ module with base $(\Omega,{\cal F}, P)$ and $P(G(S))>0$, define
$$\delta_D^{(1)}(\epsilon)
=\bigwedge\left\{{I}_D-{I}_D\left\|\frac{x+y}{2}\right\|:x, y\in S(1)\textmd{ with } B_{xy}\supset D\textmd{ and
}{I}_D\|x-y\|=\epsilon{I}_D\right\},$$
$$\delta_D^{(2)}(\epsilon)=\bigwedge\left\{{I}_D-{I}_D\left\|\frac{x+y}{2}\right\|:x, y\in U(1)\textmd{ with } B_{xy}\supset D\textmd{ and
}{I}_D\|x-y\|\geq\epsilon{I}_D\right\}$$
for any
$D\in\widetilde{\cal F}$ with $D\subset G(S)$ and $P(D)>0$ and $\epsilon\in L^{0}_{+}$ such that
$0<\epsilon\leq 2$ on $D$, where $U(1)=\{x\in S:\|x\|\leq 1\}$.
Then $\delta_D(\epsilon)=\delta^{(1)}_D(\epsilon)=\delta^{(2)}_D(\epsilon)$, where $\delta_D(\epsilon)$ is given by (1.1).
\end{theorem}

Since an $RN$ module possesses the more complicated stratification structure than a normed space, those classical proofs of (1.2) in \cite{Mgginson,Linden} do not apply to our random setting. Fortunately, Yang and Zuo recently proposed a nice new proof of (1.2) in \cite{YZ} by skillfully utilizing the classical intermediate value theorem for continuous real-valued functions and the classical Hahn-Banach theorem. Since the Hahn-Banach theorem for a.s. bounded random linear functionals is available, which is combined with Theorem \ref{thm:IVT} so that we can complete the proof of Theorem \ref{thm:expr}.

In particular, in this paper we find that the essence of Condition $(\triangle)$:
$$\{\epsilon{I}_{H(S)}:\epsilon\in R \mbox{~and~}0<\epsilon\leq 2\}\subset\{{I}_{H(S)}\|x-y\|:x,y\in S\textmd{ and }\|x\|=\|y\|={I}_{H(S)}\}$$
introduced in \cite[Section 4]{Guo-Zeng} is that the $RN$ module $(S,\|\cdot\|)$ in consideration has quasi-rank not less than 2. This fact is not only useful in the proof of Theorem \ref{thm:expr} but also leads us directly to the third main result below:

\begin{theorem}\label{thm:con}
A complete random normed module $(S,\|\cdot\|)$ is random uniformly convex iff the Banach space
$(L^{p}(S),\|\cdot\|_p)$ derived from $S$ is uniformly convex for each fixed positive number $p$ such that $1<p<+\infty$.
\end{theorem}

Theorem \ref{thm:con} was mentioned in \cite{Guo-Zeng} without a complete proof, which both implies \cite[Theorem 4.3]{Guo-Zeng} and improves \cite[Theorem 4.4]{Guo-Zeng} in that Condition ($\triangle$) in \cite[Theorem 4.4]{Guo-Zeng} has been removed.

The remainder of this paper is organized as follows: Section 2 is devoted to the detailed discussion of the notions of $L^{0}-$independence and quasi-rank in real $RN$ modules; Section 3 will prove the three main results above.

\section{Preliminaries}

First, Lemma \ref{lem:RNprop} below summarizes some basic and known facts on random conjugate spaces of $RN$ modules, whose proofs and the notion of random conjugate spaces can be found in \cite{Guo-Zeng,Guotx-compre,Guotx-extension}.

\begin{lemma}\label{lem:RNprop}
Let $(S,\|\cdot\|)$ be an $RN$ module over $K$ with base $(\Omega,{\cal F},P)$ and $(S^{\ast},\|\cdot\|^{\ast})$ its random conjugate space, then the following hold:\\
\indent(1) For any $\{x_n,n\in N\}\subset S$ and $x\in S$, $x_n\rightarrow x(n\rightarrow\infty)\Leftrightarrow\|x_n-x\|\xrightarrow{P}0(n\rightarrow\infty)$ (convergence in probability);\\
\indent(2) $S$ is a topological module
over the topological algebra $L^{0}({\cal F},K)$, namely the module
multiplication $\cdot:L^{0}({\cal F},K)\times S\rightarrow S$ is
jointly continuous;\\
\indent(3) A mapping $f:S\rightarrow L^{0}({\cal F},K)$ is continuous iff $f(x_n)\xrightarrow{P}f(x)(n\rightarrow\infty)$ for any $\{x_n,n\in N\}\subset S$ and $x\in S$ such that $x_n\rightarrow x(n\rightarrow\infty)$;\\
\indent(4) $|f(x)|\leq \|f\|^{\ast}\cdot \|x\|,\forall f\in S^{\ast}$ and $x\in S$;\\
\indent(5) $f\in S^{\ast}$ iff $f$ is a continuous module homomorphism from $S$
to $L^{0}({\cal F},K)$, i.e., $f$ is a continuous mapping from $S$ to
$L^{0}({\cal F},K)$ and $f(\xi\cdot x+\eta\cdot y)=\xi\cdot
f(x)+\eta\cdot f(y),\forall \xi, \eta\in L^{0}({\cal F},K)$ and
$x, y\in S$;\\
\indent(6) There exists a sequence $\{x_n, n\in N\}$ in the
random unit sphere $S(1)$ such that $\{\|x_n\|, n\in N\}$ converges
to ${I}_{H(S)}$ in a nondecreasing way. Further, if $(S,\|\cdot\|)$
is complete then there exists an element $x$ in $S(1)$ such that
$\|x\|={I}_{H(S)}$;\\
\indent(7) Let $(S^{\ast\ast},\|\cdot\|^{\ast\ast})$ be the random random conjugate space of $(S^{\ast},\|\cdot\|^{\ast})$ and $J: S\rightarrow S^{\ast\ast}$ a mapping defined by $J(x)(f)=f(x),\forall f\in
S^{\ast}$ and $x\in S$, then $\|J(x)\|^{\ast\ast}=\|x\|,\forall x\in S$. Such a
mapping $J$ is called {\it the canonical
embedding mapping} from $S$ to $S^{\ast\ast}$;\\
\indent(8) $\|f\|^{*}= \vee\{|f(x)|:x\in S(1)\}$ for any $f\in S^{\ast}$, further, if $(S,\|\cdot\|)$ is a real $RN$ module, then $\|f\|^{*}= \vee\{f(x):x\in S(1)\}$, so that $\|x\|=\|J(x)\|^{\ast\ast}=\vee\{f(x):f\in S^{\ast}(1)\}$.
\end{lemma}

The Hahn-Banach theorem---Theorem \ref{thm:HahnB} below for a.s. bounded random linear functionals plays an important role in the proof of Theorem \ref{thm:expr}.

\begin{theorem}[Guo \cite{Guotx-compre,Guotx-extension}]\label{thm:HahnB}
Let $(S,\|\cdot\|)$ be an $RN$ space over $K$ with base $(\Omega,{\cal
F},P)$, $M\subset S$ a linear subspace, and $f:M\rightarrow
L^{0}({\cal F},K)$ an a.s. bounded random linear functional on $M$.
Then there exists an $F\in S^{\ast}$ such that $F(x)=f(x),\forall x\in M$ and $\|F\|^{\ast}=\|f\|^{\ast}$.
As a consequence, for any $x\in S$, there exists $g\in S^{\ast}$ such that $g(x)=\|x\|$ and $\|g\|^{\ast}=I_{A_x}$, where $A_x=[\|x\|>0]$.
\end{theorem}

In the sequel, every $RN$ module $(S,\|\cdot\|)$ is assumed to have nontrivial support, namely $P(H(S))>0$.
The notions of $L^{0}-$independence and quasi-rank essentially come from \cite{Guo-Peng}.

\begin{definition}\label{def:nizhi}
 Let $(S,\|\cdot\|)$ be a real $RN$
module with base $(\Omega,{\cal F},P)$ and $D\in\widetilde{\cal F}$ such that $D\subset H(S)$ and $P(D)>0$.\\
\indent (1) For any $x,y\in S$ and $F\in\widetilde{\cal F}$, $x$ and $y$ are called {\it $L^{0}-$independent on $F$} if $\xi I_F=\eta I_F=0$ whenever $\xi,\eta\in L^{0}({\cal F},R)$ such that $\xi I_Fx+\eta I_Fy=\theta$;\\
\indent (2) If there exist $x,y\in S$ such that $x$ and $y$ are $L^{0}-$independent on $D$, then $S$ is said to {\it have quasi-rank not less than 2 on $D$} (briefly, $\textmd{Rank}_D(S)\geq 2$), otherwise $S$ is said to {\it have quasi-rank strictly less than 2 on $D$} (briefly, $\textmd{Rank}_D(S)< 2$). In particular, when $\textmd{Rank}_{H(S)}(S)\geq 2$, we simply say that $S$ has quasi-rank not less than 2, denoted by $\textmd{Rank}(S)\geq 2$.
\end{definition}

It should be mentioned that $L^{0}-$independence of three or more elements can be defined in the same manner as that of two elements. Let $(S,\|\cdot\|)$ and $D$ be the same as in Definition \ref{def:nizhi} and $x,y,z\in S$. It is easy to see that the independence of $x,y$ and $z$ on $D$ implies that of $x$ and $y$ on $D$. In addition, if $E\in\widetilde{\cal F}$ is such that $E\subset D$ and $P(E)>0$, then the $L^{0}-$independence of $x$ and $y$ on $D$ implies that on $E$, thus $\textmd{Rank}_D(S)\geq 2$ implies $\textmd{Rank}_E(S)\geq 2$.

\begin{proposition}\label{prop:duli}
Let $(S,\|\cdot\|)$ be an $RN$ module over $R$ with base $(\Omega,{\cal F},P), E\in\widetilde{\cal F}$ with $P(E)>0$ and $x,y\in S$. \\
\indent(1) If $P(A_{xy})>0$ and $x$ and $y$ are not $L^{0}-$independent on $A_{xy}$, then there exists a unique $F\in\widetilde{\cal F}$ with $F\subset A_{xy}$ and $P(F)>0$, and $\xi,\eta\in L^{0}({\cal F},R)$ with $F\subset [\xi\neq 0]\cap[\eta\neq 0]$ such that $\xi I_Fx+\eta I_F y=\theta$ and $x$ and $y$ are $L^{0}-$independent on $A_{xy}\backslash F$ whenever $P(A_{xy}\backslash F)>0$ (such $A_{xy}\backslash F$ is called {\it the $L^{0}-$independent part} of $x$ and $y$ no matter whether $P(A_{xy}\backslash F)>0$ or not). In addition, if $x$ and $y$ are $L^{0}-$independent on $A_{xy}$ their $L^{0}-$independent part is just the whole $A_{xy}$.\\
\indent(2) If $x$ and $y$ are $L^{0}-$independent on $E$, then $P(A_{xy})>0$ and $E\subset A_{xy}\backslash F$, where $F$ is the same as in (1).
\end{proposition}

\begin{proof} (1). Denote ${\cal B}=\{E\in\widetilde{\cal F}~|~E\subset A_{xy}\textmd{ such that }P(E)>0\textmd{ and }I_E\xi x+I_E\eta y=\theta$ for some $\xi,\eta\in L^{0}({\cal F},R)\textmd{ with }\xi,\eta\neq 0\textmd{ on }E\}$, where $\eta\neq 0$ on $E$ means that $E\subset [\eta\neq 0]$. Since $x,y$ are not $L^{0}-$independent on $A_{xy}$, there exist $\xi_0,\eta_0\in L^{0}({\cal F},R)$ with $\xi_0I_{A_{xy}}\neq 0$ or $\eta_0I_{A_{xy}}\neq 0$ such that $\xi_0I_{A_{xy}}x+\eta_0I_{A_{xy}}y=\theta$. One can see that $[\xi_0=0]\cap A_{xy}=[\eta_0=0]\cap A_{xy}$ by noticing that $\|x\|,\|y\|\neq 0$ on $A_{xy}$. Let $E_0=[\xi_0\neq 0]\cap A_{xy}$, then $E_0\in {\cal B}$, which shows that ${\cal B}$ is nonempty. Consequently, there exists a sequence $\{B_n,n\in N\}$ in ${\cal B}$ such that $\vee_{n\geq 1}I_{B_n}=\vee\{I_B:B\in{\cal B}\}$, namely, $I_{\cup_{n\geq 1}B_n}=\vee\{I_B:B\in{\cal B}\}$. It is clear that $B\subset \cup_{n\geq 1}B_n$ for any $B\in{\cal B}$. We will show that $F\triangleq \cup_{n\geq 1}B_n$ is just desired.

Take $\xi_n,\eta_n\in L^{0}({\cal F},R)$ such that $\xi_n,\eta_n\neq 0$ on $B_n$ and $I_{B_n}\xi_nx+I_{B_n}\eta_ny=\theta$ for any $n\in N$, and denote $E_1=B_1,E_n=B_n\backslash(\cup_{i=1}^{n-1}B_i),\forall n\geq 2$, then $\sum_{n=1}^{\infty}E_n=\cup_{n=1}^{\infty}B_n=F$. Since $P(\sum_{n=1}^{\infty}E_n)\leq 1$ and $L^{0}({\cal F},R)$ is complete, we know that $\{\sum_{n=1}^{\infty}\xi_nI_{E_n},k\in N\}$ converges in $P$ to some element $\xi$ in $L^{0}({\cal F},R)$. Further, $\xi I_{E_n}=\xi_n I_{E_n},\eta I_{E_n}=\eta_n I_{E_n},\forall n\in N$, so that $\xi,\eta\neq 0$ on $F$ and by the continuity of module multiplication, $\xi I_Fx+\eta I_Fy=\theta$.

On the other hand, $x$ and $y$ are $L^{0}-$independent on $A_{xy}\backslash F$ whenever $P(A_{xy}\backslash F)>0$. Otherwise there exists $D\in\widetilde{\cal F}$ with $D\subset A_{xy}\backslash F$ and $P(D)>0$ such that $D\in {\cal B}$, hence $D\subset F$, which is impossible.

(2). First, $E\subset A_{xy}$, otherwise $D\triangleq E\cap A_{xy}^c$ is such that $P(D)>0$. Notice that $A_{xy}^c=(A_x\backslash A_y)\cup(A_y\backslash A_x)\cup (A_x\cup A_y)^c$, we have $D=(D\cap(A_x\backslash A_y))\cup(D\cap(A_y\backslash A_x))\cup (D\cap(A_x\cup A_y)^c)$. Take $\xi_1=I_{A_y\backslash A_x}+I_{(A_x\cup A_y)^c}$ and $\eta_1=I_{A_x\backslash A_y}+I_{(A_x\cup A_y)^c}$, then clearly $\xi_1I_Dx+\eta_1I_Dy=\theta$ with $\xi_1I_D\neq 0$ or $\eta_1I_D\neq 0$, namely, $x$ and $y$ are not $L^{0}-$independent on $D$, so that they are not $L^{0}-$independent on $E$, which is a contradiction. Next, $E\subset F^c$, otherwise $G\triangleq E\cap F$ is such that $P(G)>0$, but by (1) $\xi I_Fx+\eta I_Fy=\theta$ and $\xi,\eta\neq 0$ on $F$, which implies that $\xi I_Gx+\eta I_Gy=\theta$. This is also a contradiction to the $L^{0}-$independence of $x$ and $y$ on $E$.
\end{proof}

\begin{proposition}\label{prop:G(S)}
Let $(S,\|\cdot\|)$ be a complete $RN$ module over $R$ with base $(\Omega,{\cal F},P)$ and $x_0$ as obtained in Lemma \ref{lem:RNprop}(6) such that $\|x_0\|=I_{H(S)}$. Then the following hold:\\
\indent (1) If $E\in\widetilde{\cal F}$ with $P(E)>0$ and $E\subset H(S)$ is such that $\textmd{Rank}_D(S)<2$ for any $D\in\widetilde{\cal F}$ with $P(D)>0$ and $D\subset E$, then for each $y\in S$ with $A_y\subset E$ there exists $\xi\in L^{0}({\cal F},R)$ such that $y=\xi x_0$;\\
\indent (2) If $\textmd{Rank}_E(S)\geq 2$ for some $E\in\widetilde{\cal F}$ with $E\subset H(S)$ and $P(E)>0$, then there exists a unique $G(S)\in\widetilde{\cal F}$ with $E\subset G(S)\subset H(S)$ and $P(G(S))>0$ such that $\textmd{Rank}_{G(S)}(S)\geq 2$ and $\textmd{Rank}_D(S)<2$ for any $D\in\widetilde{\cal F}$ with $D\subset H(S)\backslash G(S)$ and $P(D)>0$.
\end{proposition}

\begin{proof} (1). Suppose that $y\in S$ with $A_y\subset E$ and $P(A_y)>0$, and that the $L^{0}-$independent part of $y$ and $x_0$ as determined by Proposition \ref{prop:duli}(1) is $A_y\backslash F$. If $P(A_y\backslash F)=0$, then $y=I_{A_y}y=\xi I_{A_y}x_0$ for some $\xi\in L^{0}({\cal F},R)$. If $P(A_y\backslash F)>0$, then $\textmd{Rank}_{A_y\backslash F}(S)\geq 2$, which is a contradiction.

(2). Denote ${\cal G}=\{D\in\widetilde{\cal F}~|~D\subset H(S), P(D)>0\textmd{ and }\textmd{Rank}_D(S)\geq 2\}$, then $E\in{\cal G}$ and there exists a sequence $\{G_n,n\in N\}$ in ${\cal G}$ such that $I_{\cup_{n=1}^{\infty}G_n}=\vee\{I_D:D\in {\cal G}\}$. We will verify that $G(S)\triangleq\cup_{n=1}^{\infty}G_n$ is just desired.

Let $E_1=G_1,E_n=G_n\backslash(\cup_{i=1}^{n-1}G_i),\forall n\geq 2$, $x_n$ and $y_n$ be two elements $L^{0}-$independent on $G_n$ for each $n\in N$. Then $\sum_{n=1}^{\infty}E_n=G(S)$, and $x\triangleq\sum_{n=1}^{\infty}I_{E_n}x_n$ and $y\triangleq\sum_{n=1}^{\infty}I_{E_n}y_n$ exist by the completeness of $S$, further $I_{E_n}x=I_{E_n}x_n$ and $I_{E_n}y=I_{E_n}y_n$ for each $n\in N$. Consequently, it is easy to see that $x$ and $y$ are $L^{0}-$independent on $G(S)$. On the other hand, if there is some $D_1\in\widetilde{\cal F}$ with $D_1\subset H(S)\backslash G(S)$ and $P(D_1)>0$ such that $\textmd{Rank}_{D_1}(S)\geq 2$, then $D_1\in {\cal G}$, which yields $D_1\subset G(S)$, a contradiction.
\end{proof}

For any complete $RN$ module $(S,\|\cdot\|)$ over $R$ with base $(\Omega,{\cal F},P)$, there are only two cases that may occur:\\
\indent Case (1): There exists an $E\in\widetilde{\cal F}$ with $P(E)>0$ and $E\subset H(S)$ such that $\textmd{Rank}_E(S)\geq 2$;\\
\indent Case (2): $\textmd{Rank}_D(S)<2$ for any $D\in\widetilde{\cal F}$ with $D\subset H(S)$ and $P(D)>0$.\\
In the sequel, when Case (1) occurs $G(S)$ is always understood as in Proposition \ref{prop:G(S)}(2), at which time $P(G(S))>0$, whereas Case (2) occurs we have $P(G(S))=0$. By Proposition \ref{prop:G(S)}(1) we have the following:

\begin{corollary}\label{cor:mod}
$\delta_{H(S)\backslash G(S)}(\epsilon)=I_{H(S)\backslash G(S)}$ for any $\epsilon\in L^{0}_+$ with $0<\epsilon\leq 2$ on $H(S)\backslash G(S)$.
\end{corollary}

\begin{proposition}\label{prop:L0}
Let $(S,\|\cdot\|)$ be a complete $RN$ module over $R$ with base $(\Omega,{\cal F},P)$ and $P(G(S))>0$, and $u\in S$ with $\|u\|=I_{G(S)}$. Then there exists $v\in S$ with $\|v\|=I_{G(S)}$ such that $u$ and $v$ are $L^{0}-$independent on $G(S)$.
\end{proposition}

\begin{proof} Since $\textmd{Rank}_{G(S)}(S)\geq 2$, we can take a pair of elements $x,y\in S$ with $\|x\|=\|y\|=I_{G(S)}$ such that $x$ and $y$ are $L^{0}-$independent on $G(S)$. Denote ${\cal B}=\{E\in\widetilde{\cal F}~|~E\subset G(S)$ and there exist $\xi,\eta\in L^{0}({\cal F},R)\textmd{ such that }I_Eu=\xi I_Ex+\eta I_Ey\}$. Then our proof is divided into the following two cases.

Case (1): when ${\cal B}=\{\tilde{\emptyset}\}$, it is easy to see that $u,x$ and $y$ are $L^{0}-$independent on $G(S)$, thus $v:=x$ or $y$ is just desired.

Case (2): otherwise, there exists a sequence $\{B_n,n\in N\}$ in ${\cal B}$ such that $I_{\cup_{n\geq 1}B_n}=\vee_{n\geq 1}I_{B_n}=\vee\{I_B:B\in{\cal B}\}$ and $\xi_n,\eta_n\in L^{0}({\cal F},R)$ such that $I_{B_n}u=\xi_nI_{B_n}x+\eta_nI_{B_n}y$ for each $n\in N$. Set $E_1=B_1,E_n=B_n\backslash \cup_{i=1}^{n-1}B_i,\forall n\geq 2$, then $\{\sum_{n=1}^{k}\xi_nI_{E_n},k\in N\}$ (accordingly, $\{\sum_{n=1}^{k}\eta_nI_{E_n},k\in N\}$) converges in $P$ to some element $\xi$ (accordingly, $\eta$) in $L^{0}({\cal F},R)$. Furthermore, $\xi I_{E_n}=\xi_n I_{E_n}$ and $\eta I_{E_n}=\eta_n I_{E_n},\forall n\in N$, then by the continuity of module multiplication,
$$I_Du=\xi I_Dx+\eta I_D y,\eqno(2.1)$$
where $D\triangleq\sum_{n=1}^{\infty}E_n=\cup_{n=1}^{\infty}B_n\subset G(S)$. It is easy to prove by way of contradiction that $u,x,$ and $y$ are $L^{0}-$independent on $G(S)\backslash D$.

Denote $F_{\xi}=D\cap[\xi=0]$ and $F_{\eta}=D\cap[\eta=0]$, then from (2.1) it follows that $I_{F_{\xi}}u=\eta I_{F_{\xi}}y$, which implies $I_{F_{\xi}}|\eta|=I_{F_{\xi}}$. Similarly, $I_{F_{\eta}}u=\xi I_{F_{\eta}}x$ and $I_{F_{\eta}}|\xi|=I_{F_{\eta}}$.

Let $v=I_{D\backslash(F_{\xi}\cup F_{\eta})}x+I_{F_{\xi}\cup F_{\eta}}\|x+y\|^{-1}(x+y)+I_{G(S)\backslash D}x$, clearly $\|v\|=I_{G(S)}$. Now we suppose that $k_1,k_2\in L^{0}({\cal F},R)$ are such that $k_1u+k_2v=\theta$, namely,
$$k_1(I_{G(S)\backslash D}u+\xi I_Dx+\eta I_D y)+k_2(I_{D\backslash(F_{\xi}\cup F_{\eta})}x+I_{F_{\xi}\cup F_{\eta}}\|x+y\|^{-1}(x+y)+I_{G(S)\backslash D}x)=\theta.\eqno(2.2)$$
Clearly, multiplying both sides of (2.2) by $I_{G(S)\backslash D}$ yields $k_1I_{G(S)\backslash D}u+k_2I_{G(S)\backslash D}x=\theta$, thus $k_1=k_2=0$ on $G(S)\backslash D$. In the same way, we can verify that $k_1=k_2=0$ on $F_{\xi}$, $F_{\eta}$ and $D\backslash(F_{\xi}\cup F_{\eta})$, respectively. Therefore, $k_1=k_2=0$ on $G(S)$.
\end{proof}

\section{Proofs of the main results}
We can now prove Theorem \ref{thm:IVT}, the idea of whose proof is very similar to that of \cite[Lemma 4.7]{KV}, but since Theorem \ref{thm:IVT} is of crucial importance in this paper, we give its proof in detail.

{\noindent\bf Proof of Theorem \ref{thm:IVT}.} It suffices to prove the special case when $f(Y_1)\leq f(Y_2)$ and $Y_1\leq Y_2$: since, otherwise, let $C=[f(Y_1)\leq f(Y_2)]$ and $D=[f(Y_1)>f(Y_2)]$, then for $C$ we apply the special case to $f_C=I_Cf$ and $\xi_C=I_C\xi$ so that we can obtain $\eta_C\in[Y_1,Y_2]$ such that $I_Cf(\eta_C)=I_C\xi$; for $D$ we apply the special case to $f_D=-I_Df$ and $\xi_D=-I_D\xi$ so that we can obtain $\eta_D\in[Y_1,Y_2]$ such that $I_Df(\eta_D)=I_D\xi$, consequently, $\eta:=I_C\eta_C+I_D\eta_D$ satisfies our requirements. In the following, we will give the proof of the special case.

Denote $A_1=[\xi=f(Y_1)], A_2=[\xi=f(Y_2)]$ and $H=\tilde{\Omega}\backslash(A_1\cup A_2)$, where $\tilde{\Omega}$ is the equivalence class of $\Omega$. If we can find some $\eta\in[Y_1,Y_2]$ such that $f(\eta)=\xi I_H$ then $f(I_{A_1}Y_1+I_{A_2\backslash A_1}Y_2+I_H\eta)=I_{A_1}\xi+I_{A_2\backslash A_1}\xi+I_H\xi=\xi$. Thus we can suppose that $f(Y_1)<\xi< f(Y_2)$ on $\tilde{\Omega}$, which certainly implies $Y_1<Y_2$ on $\tilde{\Omega}$.

Let $\eta=\wedge{\cal G}$, where ${\cal G}=\{Y\in [Y_1,Y_2]~|~f(Y)\geq \xi\textmd{ and }Y\geq Y_1\}$.
Clearly, $Y_2\in {\cal G}$, which shows that ${\cal G}$ is non-void and $\eta\in [Y_1,Y_2]$. Since ${\cal G}$ is dually directed, there is a sequence $\{W_n,n\in N\}\subset {\cal G}$ such that $W_n\searrow\eta(n\rightarrow\infty)$, which together with the continuity of $f$ implies that $f(W_n)\xrightarrow{P}f(\eta)(n\rightarrow\infty)$. Thus $f(\eta)\geq\xi$. We will further prove that $f(\eta)=\xi$ as follows.

Assume by way of contradiction that there exists some $E\in\widetilde{\cal F}$ with $P(E)>0$ such that $f(\eta)>\xi$ on $E$. Let $U_n=(\eta-1/n)\vee Y_1$ for each $n\in N$, then $U_n\rightarrow \eta(n\rightarrow\infty)$ by noticing that $\eta\geq Y_1$. We claim that $f(U_n)<\xi$ on $\tilde{\Omega}$ for each $n\in N$. Otherwise, there exists some $i\in N$ such that $P(D_i)>0$, where $D_i=[f(U_i)\geq\xi]$. Denote $B_i=[\eta-1/i>Y_1]$, then $P(D_i\cap B_i^c)=0$ (otherwise $I_{D_i\cap B_i^c}f(Y_1)=I_{D_i\cap B_i^c}f(U_i)\geq\xi I_{D_i\cap B_i^c}$, which is a contradiction), i.e., $D_i=D_i\cap B_i$. Further, $I_{D_i}f(\eta-1/i)=I_{D_i\cap B_i}f(U_i)=I_{D_i}f(U_i)$ and $\eta-1/i\geq Y_1$ on $D_i$, which yields $f(I_{D_i}(\eta-1/i)+I_{\tilde{\Omega}\backslash D_i}Y_2)\geq\xi$, and hence $I_{D_i}(\eta-1/i)+I_{D_i^c}Y_2\in{\cal G}$, which in turn implies  $I_{D_i}(\eta-1/i)+I_{D_i^c}Y_2\geq\eta$, but this is impossible.
Thus $P(D_n)=0$, namely, $f(U_n)<\xi$ on $\tilde{\Omega}$ for any $n\in N$.

Observing $I_Ef(U_n)<\xi I_E<f(\eta)I_E$ on $E$ and recalling $U_n\rightarrow \eta$, by the continuity of $f$ we have $I_Ef(\eta)\leq\xi I_E<f(\eta)I_E$ on $E$, which is an obvious contradiction. Therefore, $f(\eta)=\xi$.\quad$\Box$

Lemma \ref{lem:L3.1} below together with Proposition \ref{prop:L0} is a preparation for the proof of Proposition \ref{prop:hard} below that is key to the proof of Theorem \ref{thm:expr}.

\begin{lemma}\label{lem:L3.1}
Let $(S,\|\cdot\|)$ be an $RN$ module over $R$ with base $(\Omega,{\cal F},P)$, $E\in\widetilde{\cal F}$ with $P(E)>0$ and $x,y\in S$ with $\|x\|=I_{A_{xy}}$ and $\|y\|\leq 1$. If $x$ and $y$ are $L^{0}-$independent on $E$, then there exist $u_E,v_E\in S$ such that $\|u_E\|=\|v_E\|=I_E$ and $u_E-v_E=I_E(x-y)$.
\end{lemma}

\begin{proof} By Proposition \ref{prop:duli}(2) we can see that $E\subset A_{xy}$. Define a mapping $f_E:L^{0}({\cal F}, R)\rightarrow L^{0}({\cal F}, R)$ by
$$f_E(\alpha)=I_E\left\|\frac{(\cos\alpha)x+(\sin\alpha)y}{\|(\cos\alpha)x+(\sin\alpha)y\|}-x+y\right\|,\forall\alpha\in L^{0}({\cal F}, R),\eqno(3.1)$$
then it is clear that $f_E$ is a continuous local function. By the $L^{0}-$independence of $x$ and $y$ on $E$ we can see that $\|(\cos\alpha)x+(\sin\alpha)y\|>0$ on $E$. Since $f_E(0)=I_E\|y\|\leq I_E$ and $f_E(3\pi/4)=I_E(1+\|x-y\|)\geq I_E$, by Theorem \ref{thm:IVT} there exists $\eta_E\in L^{0}({\cal F}, R)$ with $0\leq\eta_E\leq3\pi/4$ such that $f_E(\eta_E)=I_E$. Denote $$u_E=I_E\frac{(\cos\eta_E)x+(\sin\eta_E)y}{\|(\cos\eta_E)x+(\sin\eta_E)y\|}\eqno(3.2)$$
and $v_E=I_E(u_E-x+y)$, then $u_E$ and $v_E$ are desired. Indeed, $\|u_E\|=I_E$, $\|v_E\|=f_E(\eta_E)=I_E$ and $u_E-v_E=I_Eu_E-I_E(u_E-x+y)=I_E(x-y)$.
\end{proof}

\begin{proposition}\label{prop:hard}
Let $(S,\|\cdot\|)$ be a complete $RN$ module over $R$ with base $(\Omega,{\cal F},P)$ and $P(G(S))>0$, and $x,y\in S$ with $P(A_{xy})>0,A_{xy}\subset G(S),\|x\|=I_{A_{xy}}$ and $\|y\|\leq 1$. Then there exist $u,v\in S$ satisfying the following two conditions:\\
\indent (1) $\|u\|=\|v\|=I_{A_{xy}}$ and $u-v=I_{A_{xy}}(x-y)$,\\
\indent (2) $\|u+v\|\geq I_{A_{xy}}\|x+y\|$.
\end{proposition}

\begin{proof} It is divided into two Steps; Step 1 is to obtain $u$ and $v$ satisfying (1), and Step 2 to show that these two elements $u$ and $v$ also satisfy (2).

(Step 1): First of all, there exist $u_0,v_0$ in $S$ with $\|u_0\|=\|v_0\|=I_{G(S)}$ such that $u_0$ and $v_0$ are $L^{0}-$independent on $G(S)$ by the meaning of $G(S)$.

Let $F,\xi$ and $\eta$ be the same as obtained in Proposition \ref{prop:duli}(1) with respect to $x$ and $y$ (if $x$ and $y$ are $L^{0}-$independent on $A_{xy}$, then $F$ is $\tilde{\emptyset}$) and denote $G=A_{xy}\backslash F, G_1=G\cap [\|y\|<1],F_1=F\cap [\|y\|<1]$ and $E_0=A_{xy}\cap[\|y\|=1]$. Clearly $A_{xy}=G_1\cup F_1\cup E_0$ since $\|y\|\leq 1$. We will deal with the problem on $G_1,F_1$ and $E_0$, respectively. Without loss of generality, we can assume that $G_1,F_1$ and $E_0$ all have positive probabilities.

Considering $x$ and $y$ on $G_1$, we obtain $u_{G_1}$ and $v_{G_1}$ in $S$ such that $\|u_{G_1}\|=\|v_{G_1}\|=I_{G_1}$ and $u_{G_1}-v_{G_1}=I_{G_1}(x-y)$ by Lemma \ref{lem:L3.1}.

Then consider the case on $F_1$. From $\xi I_Fx+\eta I_Fy=0$ and $F\subset[\xi\neq 0]\cap[\eta\neq 0]$ it follows that $I_Fy=\gamma I_Fx$, where $\gamma=-\xi\eta^{-1}\neq 0$ on $F$, further $|\gamma|I_F=I_F\|y\|\leq I_F$, so that $F_1=F\cap[|\gamma|<1]=F\cap[0<|\gamma|<1]$ and let $F_{11}=F\cap[0<\gamma<1]$ and $F_{12}=F\cap[-1<\gamma<0]$, then $F_1=F_{11}\cup F_{12}$.

Let $x^{\prime}=I_{A_{xy}}x+I_{G(S)\backslash A_{xy}}u_0$, then $\|x^{\prime}\|=I_{G(S)}$, and by Proposition \ref{prop:L0} there exists $x^{\prime\prime}\in S$ with $\|x^{\prime\prime}\|=I_{G(S)}$ such that $x^{\prime}$ and $x^{\prime\prime}$ are $L^{0}-$independent on $G(S)$, which also implies that $x$ and $x^{\prime\prime}$ are $L^{0}-$independent on both $F_{11}$ and $F_{12}$.

Let $$w_1=I_{F_{11}}\frac{1-\gamma}{2}x^{\prime\prime},~~~x_1=I_{F_{11}}\frac{w_1+x}{\|w_1+x\|},~~~y_1=I_{F_{11}}y+x_1-I_{F_{11}}x,$$
then it is easy to see that $x$ and $w_1$, further $x$ and $x_1$ are $L^{0}-$independent on $F_{11}$, so that $x_1$ and $y_1$ are $L^{0}-$independent on $F_{11}$ by noticing that $y_1=(\gamma-1)I_{F_{11}}x+x_1$. By Proposition \ref{prop:duli}(2) $F_{11}\subset A_{x_1y_1}$, which together with the obvious fact that $A_{x_1}=F_{11}$ and $A_{y_1}\subset F_{11}$ implies that $A_{x_1y_1}=F_{11}$.
Besides,
\begin{eqnarray*}
I_{F_{11}}\|x_1-x\|&=&I_{F_{11}}\left\|\frac{w_1+x}{\|w_1+x\|}-x-w_1+w_1\right\|\\
&\leq&I_{F_{11}}\left\|\frac{w_1+x}{\|w_1+x\|}-x-w_1\right\|+I_{F_{11}}\|w_1\|\\
&=&I_{F_{11}}\big|\|x\|-\|w_1+x\|\big|+I_{F_{11}}\|w_1\|\\
&\leq& 2I_{F_{11}}\|w_1\|\\
&=&I_{F_{11}}(1-\gamma),
\end{eqnarray*}
from which it follows that $I_{F_{11}}\|y_1\|=I_{F_{11}}\|\gamma x+x_1-x\|\leq I_{F_{11}}\gamma+I_{F_{11}}(1-\gamma)=I_{F_{11}}.$

Applying Lemma \ref{lem:L3.1} to $x_1$ and $y_1$ on $F_{11}$, we obtain $u_{F_{11}},v_{F_{11}}\in S$ with $\|u_{F_{11}}\|=\|v_{F_{11}}\|=I_{F_{11}}$ such that $u_{F_{11}}-v_{F_{11}}=I_{F_{11}}(x_1-y_1)=I_{F_{11}}(x-y)$.

In the same way, let  $w_2=I_{F_{12}}((1+\gamma)/2)x^{\prime\prime}$, $x_2=I_{F_{12}}\|w_2+x\|^{-1}(w_2+x)$ and consider $x_2$ and $y_2:=I_{F_{12}}y+x_2-I_{F_{12}}x$, we can obtain $u_{F_{12}},v_{F_{12}}\in S$ with $\|u_{F_{12}}\|=\|v_{F_{12}}\|=I_{F_{12}}$ such that $u_{F_{12}}-v_{F_{12}}=I_{F_{12}}(x-y)$.

As for the case on $E_0$, let $u_{E_0}=I_{E_0}x$ and $v_{E_0}=I_{E_0}y$, then $\|u_{E_0}\|=\|v_{E_0}\|=I_{E_0}$ and $u_{E_0}-v_{E_0}=I_{E_0}(x-y)$.

Finally, let $u=u_{G_1}+u_{F_{11}}+u_{F_{12}}+u_{E_0}$ and $v=v_{G_1}+v_{F_{11}}+v_{F_{12}}+v_{E_0}$, it is easy to see that $\|u\|=\|v\|=I_{A_{xy}}$ and $u-v=I_{A_{xy}}(u-v)=I_{A_{xy}}(x-y)$ by noticing that $A_{xy}=G_1\cup F_{11}\cup F_{12}\cup E_0$.

(Step 2): Since $A_{xy}=G_1\cup F_1\cup E_0$ and $I_{E_0}\|u+v\|=I_{E_0}\|x+y\|$, we only need to prove $I_{F_1\cup G_1}\|u+v\|\geq I_{F_1\cup G_1}\|x+y\|$.

First we will find $\lambda,\beta\in L^{0}_+$ with $\lambda\geq 1$ on $F_1\cup G_1$ such that $$I_{F_1\cup G_1}x=\lambda I_{F_1\cup G_1}\frac{x+y}{2}-\beta I_{F_1\cup G_1}(u-x).\eqno(3.3)$$

Since $I_{F_1}y=\gamma I_{F_1}x$, where $0<|\gamma|<1$ on $F_1$, we can take $\lambda,\beta\in L^{0}_+$ such that $\lambda I_{F_1}=2I_{F_1}(1+\gamma)^{-1}$ and $\beta I_{F_1}=0$, which certainly implies the validity of (3.3) on $F_1$ and that $\lambda\geq 1$ and $\beta\geq 0$ on $F_1$. Further, recalling that $I_{G_1}u=u_{G_1}$ is determined by (3.2) when $E$ is replaced with $G_1$, considering (3.3) on $G_1$ we have
$$\left\{ \begin{array}{ll}
\frac{\lambda}{2}I_{G_1}&=I_{G_1}\beta\frac{\sin\eta_{G_1}}{\|(\cos\eta_{G_1})x+(\sin\eta_{G_1})y\|},\\
I_{G_1}&=\frac{\lambda}{2}I_{G_1}-I_{G_1}\beta\frac{\cos\eta_{G_1}}{\|(\cos\eta_{G_1})x+(\sin\eta_{G_1})y\|}+I_{G_1}\beta,
 \end{array} \right.$$
which yields
$$\beta I_{G_1}=I_{G_1}\frac{\|(\cos\eta_{G_1})x+(\sin\eta_{G_1})y\|}{\|(\cos\eta_{G_1})x+(\sin\eta_{G_1})y\|-\cos\eta_{G_1}+\sin\eta_{G_1}}$$
and
$$\lambda I_{G_1}=I_{G_1}\left(1-\beta+\beta\frac{\cos\eta_{G_1}+\sin\eta_{G_1}}{\|(\cos\eta_{G_1})x+(\sin\eta_{G_1})y\|}\right),$$
Notice that $f_{G_1}(\eta_{G_1})=I_{G_1}$. We will check that $\beta\geq 0$ on $G_1$ and $\lambda\geq 1$ on $G_1$ as follows.

We claim that $\sin\eta_{G_1}>0$ on $G_1$, otherwise there exists $D\in\widetilde{\cal F}$ with $D\subset G_1$ and $P(D)>0$ such that $I_D\sin\eta_{G_1}=0$, then $I_D\eta_{G_1}=0$ since $0\leq\eta_{G_1}\leq3\pi/4$, and by (3.1) we have $I_Df_{G_1}(\eta_{G_1})=I_D\|y\|$, a contradiction to the equality $f_{G_1}(\eta_{G_1})=I_{G_1}$. Further, combining the relation
\begin{eqnarray*}
I_{G_1}\|(\cos\eta_{G_1})x+(\sin\eta_{G_1})y\|
&\geq& I_{G_1}(|\cos\eta_{G_1}|-|\sin\eta_{G_1}|\|y\|)\\
&>&I_{G_1}(|\cos\eta_{G_1}|-\sin\eta_{G_1})\\
&\geq& I_{G_1}(\cos\eta_{G_1}-\sin\eta_{G_1}))\textmd{ on $G_1$}
\end{eqnarray*}
and the $L^{0}-$independence of $x$ and $y$ on $G_1$ we can see that $\beta>0$ on $G_1$.

On the other hand, applying Theorem \ref{thm:HahnB} to $I_{F_1\cup G_1}u$ we obtain some $x^{\ast}\in S^{\ast}$ such that $x^{\ast}(I_{F_1\cup G_1}u)=I_{F_1\cup G_1}$ and $\|x^{\ast}\|=I_{F_1\cup G_1}$. Thus $I_{F_1\cup G_1}\geq x^{\ast}(I_{F_1\cup G_1}v)=I_{F_1\cup G_1}x^{\ast}(u-x+y)=I_{F_1\cup G_1}-x^{\ast}(I_{F_1\cup G_1}x)+x^{\ast}(I_{F_1\cup G_1}y)$, which yields $x^{\ast}(I_{F_1\cup G_1}x)\geq x^{\ast}(I_{F_1\cup G_1}y)$. Again since $\cos\eta_{G_1}+\sin\eta_{G_1}\geq 0$, it follows from
$$I_{G_1}=x^{\ast}(I_{G_1}u)=x^{\ast}\left(I_{G_1}\frac{(\cos\eta_{G_1})x+(\sin\eta_{G_1})y}{\|(\cos\eta_{G_1})x+(\sin\eta_{G_1})y\|}\right)$$
that
\begin{eqnarray*}
I_{G_1}\|(\cos\eta_{G_1})x+(\sin\eta_{G_1})y\|&=&(\cos\eta_{G_1})x^{\ast}(I_{G_1}x)+(\sin\eta_{G_1})x^{\ast}(I_{G_1}y)\\
&\leq& x^{\ast}(I_{G_1}x)(\cos\eta_{G_1}+\sin\eta_{G_1})\\
&\leq& I_{G_1}(\cos\eta_{G_1}+\sin\eta_{G_1}),
\end{eqnarray*}
which implies that $\lambda I_{G_1}\geq I_{G_1}$.

Next, we have the following equivalent relations:
\begin{eqnarray*}
(3.3)&\Leftrightarrow& \frac{\lambda}{2}I_{F_1\cup G_1}(y+x)=\beta I_{F_1\cup G_1}u+(1-\beta)I_{F_1\cup G_1}x\\
&\Leftrightarrow& \frac{\lambda}{2}I_{F_1\cup G_1}(y-x)=\beta I_{F_1\cup G_1}u+(1-\beta-\lambda)I_{F_1\cup G_1}x\\
&\Leftrightarrow& \frac{\lambda}{2}I_{F_1\cup G_1}(v-u)=\beta I_{F_1\cup G_1}u+(1-\beta-\lambda)I_{F_1\cup G_1}x ~(\textmd{by the results in (1)})\\
&\Leftrightarrow& \frac{\lambda}{2}I_{F_1\cup G_1}(v+u)=I_{F_1\cup G_1}x+(\beta+\lambda)I_{F_1\cup G_1}(u-x),
\end{eqnarray*}
so that
\begin{equation*}
\begin{cases}I_{F_1\cup G_1}\|x+(\beta+\lambda)(u-x)\|=\frac{\lambda}{2}I_{F_1\cup G_1}\|u+v\|,\\
I_{F_1\cup G_1}\|x+\beta(u-x)\|=\frac{\lambda}{2}I_{F_1\cup G_1}\|x+y\|.
\end{cases}\eqno(3.4)
\end{equation*}

Besides, we have the following relations:
\begin{eqnarray*}
I_{F_1\cup G_1}\|x+\beta(u-x)\|&\stackrel{\textmd{(a)}}{=}&\vee\{I_{F_1\cup G_1}f(x)+\beta I_{F_1\cup G_1}f(u-x)~|~f\in S^{\ast}(1)\}\\
&\stackrel{\textmd{(b)}}{\leq}&\left(\vee\{I_{F_1\cup G_1}f(x)+(\beta+\lambda)I_{F_1\cup G_1}f(u-x)~|~f\in S^{\ast}(1)\}\right)\vee I_{F_1\cup G_1}\\
&\stackrel{\textmd{(c)}}{=}&\vee\{I_{F_1\cup G_1}f(x)+(\beta+\lambda)I_{F_1\cup G_1}f(u-x)~|~f\in S^{\ast}(1)\}\\
&\stackrel{\textmd{(d)}}{=}&I_{F_1\cup G_1}\|x+(\beta+\lambda)(u-x)\|.
\end{eqnarray*}\\
Indeed, $\textmd{(a)}$ and $\textmd{(d)}$ are clear by Lemma \ref{lem:RNprop}(8), and $\textmd{(b)}$ and $\textmd{(c)}$ can be verified respectively as follows.

Denote $F_1\cup G_1=D^{(1)}_f\cup D^{(2)}_f$ for any $f\in S^{\ast}(1)$, where $D^{(1)}_f=(F_1\cup G_1)\cap[f(x)+\beta f(u-x)>1]$ and $D^{(2)}_f=(F_1\cup G_1)\cap[f(x)+\beta f(u-x)\leq 1]$, then we have that $f(u-x)>0$ on $D^{(1)}_f$, so that $I_{D^{(1)}_f}(f(x)+\beta f(u-x))\leq I_{D^{(1)}_f}(f(x)+(\beta+\lambda) f(u-x))$, which together with the obvious fact that $I_{D^{(2)}_f}(f(x)+\beta f(u-x))\leq I_{D^{(2)}_f}$ implies the inequality $\textmd{(b)}$.

Recalling $x^{\ast}(I_{F_1\cup G_1}u)=\|x^{\ast}\|=I_{F_1\cup G_1}$ and $\beta+\lambda\geq 1$ on $F_1\cup G_1$ we have that $I_{F_1\cup G_1}x^{\ast}(x)+(\beta+\lambda)I_{F_1\cup G_1}x^{\ast}(u-x)=I_{F_1\cup G_1}(1-\beta-\lambda)x^{\ast}(x)+(\beta+\lambda)I_{F_1\cup G_1}\geq I_{F_1\cup G_1}(1-\beta-\lambda)+(\beta+\lambda)I_{F_1\cup G_1}=I_{F_1\cup G_1}$, which yields the equality $\textmd{(c)}$.

Finally, combining the previous relations with (3.4) we have $I_{F_1\cup G_1}\|u+v\|\geq I_{F_1\cup G_1}\|x+y\|$, which completes the proof.
\end{proof}

Lemma \ref{lem:L3.2} below, which reveals the essence of Condition ($\triangle$), will play important roles in the proofs of Theorem \ref{thm:expr} and Proposition \ref{prop:key}.

\begin{lemma}\label{lem:L3.2}
Let $(S,\|\cdot\|)$ be an $RN$ module over $R$ with base $(\Omega,{\cal F},P)$, then the following two statements are equivalent to each other for any $D\in\widetilde{\cal F}$ with $D\subset H(S)$ and $P(D)>0$:\\
\indent (1) $\textmd{Rank}_D(S)\geq 2$;\\
\indent (2) $\{\epsilon I_D:\epsilon\in L^{0}_+\textmd{ with }0<\epsilon\leq 2\textmd{ on }D\}\subset\{\|u-v\|:u,v\in S\textmd{ with }\|u\|=\|v\|=I_D\}$.
\end{lemma}

\begin{proof}
(2)$\Rightarrow$(1). Take an arbitrary pair $x,y\in S$ such that $\|x\|=\|y\|=I_D$ and $\|x-y\|=I_D$, we can show that $x$ and $y$ are $L^{0}-$independent on $D$. In fact, let $\xi,\eta\in L^{0}({\cal F},R)$ such that $\xi I_D x+\eta I_D y=\theta$ and suppose that there exists some $E\in\widetilde{\cal F}$ with $E\subset D$ and $P(E)>0$ such that $\xi\neq 0$ on $E$, then $I_Ex=-I_E\xi^{-1}\eta y$, which implies $I_E|\xi^{-1}\eta|=I_E$. This will lead to a contradiction by noticing that $I_E\|x-y\|=I_E$. Thus $\xi I_D=0$, and similarly $\eta I_D=0$.

(1)$\Rightarrow$(2).
Since $\textmd{Rank}_D(S)\geq 2$, we can take $x,y\in S$ with $\|x\|=\|y\|=I_{D}$ such that $x$ and $y$ are $L^{0}-$independent on $D$, and define $f:L^{0}({\cal F},R)\rightarrow L^{0}({\cal F},R)$ by
$$f(\alpha)=I_{D}\left\|\frac{(\cos\alpha)x-(\sin\alpha)y}{\|(\cos\alpha)x-(\sin\alpha)y\|}-x\right\|,\forall\alpha\in L^{0}({\cal F},R).$$
It is obvious that $f$ is a continuous local function and $\|(\cos\alpha)x-(\sin\alpha)y\|>0$ on $D$ for any $\alpha\in L^{0}({\cal F},R)$. Consequently, since $f(0)=0$ and $f(\pi)=2I_{D}$, by Theorem \ref{thm:IVT} there exists $\eta^{(\epsilon)}\in L^{0}({\cal F},R)$ with $0\leq\eta^{(\epsilon)}\leq\pi$ for any $\epsilon\in L^{0}_+\textmd{ with }0<\epsilon\leq 2\textmd{ on }D$ such that $f(\eta^{(\epsilon)})=\epsilon I_{D}$. Denote $v=I_D\|(\cos\eta^{(\epsilon)})x-(\sin\eta^{(\epsilon)})y\|^{-1}((\cos\eta^{(\epsilon)})x-(\sin\eta^{(\epsilon)})y)$, then $\|v\|=I_{D}$ and $\|x-v\|=\epsilon I_{D}$.
\end{proof}

\vspace{3mm}
We can now prove Theorem \ref{thm:expr}.

{\noindent\bf Proof of Theorem \ref{thm:expr}.} For $\epsilon$ and $D$ as assumed, since $\delta^{(2)}_D(\epsilon)\leq \delta_D(\epsilon)\leq \delta^{(1)}_D(\epsilon)$, we only need to show that $\delta^{(1)}_D(\epsilon)\leq \delta_D(\epsilon)\leq \delta^{(2)}_D(\epsilon)$. The proof is divided into two parts.

(Part 1). We show that $\delta^{(1)}_D(\epsilon)\leq \delta_D(\epsilon)$. First of all, since $\textmd{Rank}_D(S)\geq 2$, by Lemma \ref{lem:L3.2} there exist $u_0,v_0\in S$ with $\|u_0\|=\|v_0\|=I_{D}$ such that $\|u_0-v_0\|=\epsilon I_{D}$.

For any $x,y\in S(1)$ such that $B_{xy}\supset D$ and $I_D\|x-y\|\geq\epsilon I_D$, we can write $I_D\|x-y\|=b\epsilon I_D$, where $b\in L^{0}({\cal F},R)$ is such that $bI_D=I_D\|x-y\|\epsilon^{-1}$, clearly $bI_D\geq I_D$. Applying Theorem \ref{thm:HahnB} to $I_D\frac{x+y}{2}$ we obtain an element $x^{\ast}\in S^{\ast}$ such that $$x^{\ast}(I_D\frac{x+y}{2})=I_D\|\frac{x+y}{2}\|, ~~\|x^{\ast}\|=I_{D_1},\eqno(3.5)$$ where $D_1=[I_D\|\frac{x+y}{2}\|>0]$ (obviously, $D_1\subset D$).

First, let $E=[x^{\ast}(I_Dx)\geq x^{\ast}(I_Dy)]\cap D_1$, then we can obtain $u_E,v_E\in S$ with $\|u_E\|=\|v_E\|=I_E$ such that $\|u_E-v_E\|=\epsilon I_E$ and $\|u_E+v_E\|\geq I_E\|x+y\|$.

In fact, let $x_1=I_Ex,y_1=I_Ex+I_Eb^{-1}(y-x)$, then one can see that $\|x_1\|=I_E,\|y_1\|=\|I_E(1-b^{-1})x+I_Eb^{-1}y\|\leq I_E$ and $\|x_1-y_1\|=\epsilon I_E$. Again denote $E^{\prime}=[I_E\|y_1\|>0]$ (certainly, $E^{\prime}\subset E\subset D_1$) and $E^{\prime\prime}=E\backslash E^{\prime}$, then we can handle the problem on $E^{\prime\prime}$ and $E^{\prime}$, respectively.

On $E^{\prime\prime}$: from $I_{E^{\prime\prime}}\|y_1\|=0$ it follows that $I_{E^{\prime\prime}}y_1=\theta$, namely, $I_{E^{\prime\prime}}y=(1-b)I_{E^{\prime\prime}}x$, which in turn implies $bI_{E^{\prime\prime}}=2I_{E^{\prime\prime}}$, so that $I_{E^{\prime\prime}}y=-I_{E^{\prime\prime}}x$ and hence $I_{E^{\prime\prime}}\|x+y\|=0$.

On $E^{\prime}$: without loss of generality, we can suppose that $P(E^{\prime})>0$. Notice that $\|I_{E^{\prime}}x_1\|=I_{E^{\prime}}$, $\|I_{E^{\prime}}y_1\|\leq I_{E^{\prime}}$ and $A_{I_{E^{\prime}}x_1}=A_{I_{E^{\prime}}y_1}=E^{\prime}$, by Proposition \ref{prop:hard} we have two elements $u^{\prime},v^{\prime}\in S$ with $\|u^{\prime}\|=\|v^{\prime}\|=I_{E^{\prime}}$ such that $u^{\prime}-v^{\prime}=I_{E^{\prime}}(I_{E^{\prime}}x_1-I_{E^{\prime}}y_1)=I_{E^{\prime}}b^{-1}(x-y)$ and $\|u^{\prime}+v^{\prime}\|\geq I_{E^{\prime}}\|I_{E^{\prime}}x_1+I_{E^{\prime}}y_1\|=I_{E^{\prime}}\|x_1+y_1\|$.
It is clear that $\|u^{\prime}-v^{\prime}\|=\epsilon I_{E^{\prime}}$. Further, by the choice of $E^{\prime}$ and (3.5)
$I_{E^{\prime}}\|x_1+y_1\|\geq I_{E^{\prime}}x^{\ast}(x_1+y_1)
=2I_{E^{\prime}}x^{\ast}(x_1)+I_{E^{\prime}}x^{\ast}(y_1-x_1)=2I_{E^{\prime}}x^{\ast}(x)+I_{E^{\prime}}b^{-1}x^{\ast}(y-x)
\geq 2I_{E^{\prime}}x^{\ast}(x)+I_{E^{\prime}}x^{\ast}(y-x)=I_{E^{\prime}}\|x+y\|$.

Let $u_E=I_{E^{\prime}}u^{\prime}+I_{E^{\prime\prime}}u_0$ and $v_E=I_{E^{\prime}}v^{\prime}+I_{E^{\prime\prime}}v_0$, then we can see that $u_E$ and $v_E$ are just desired.

Next, denote $F=[x^{\ast}(I_Dx)<x^{\ast}(I_Dy)]\cap D_1$, by the symmetry of $x$ and $y$ we can also have $u_F,v_F\in S$ with $\|u_F\|=\|v_F\|=I_F$ such that $\|u_F-v_F\|=\epsilon I_F$ and $\|u_F+v_F\|\geq I_F\|x+y\|$.

At last, let $u=I_Eu_E+I_Fu_F+I_{D\backslash D_1}u_0$ and $v=I_Ev_E+I_Fv_F+I_{D\backslash D_1}v_0$, then we can see that $\|u\|=\|v\|=I_D,\|u-v\|=\epsilon I_D$ and $\|u+v\|\geq I_D\|x+y\|$. This shows that $\delta^{(1)}_D(\epsilon)\leq \delta_D(\epsilon)$.

(Part 2). We show that $\delta_D(\epsilon)\leq \delta^{(2)}_D(\epsilon)$. For any fixed $x,y\in S$ such that $B_{xy}\supset D,\|x\|,\|y\|\leq 1$ and $I_D\|x-y\|\geq\epsilon I_D$, let $E=[\|x\|\geq\|y\|]\cap D, F=[\|x\|<\|y\|]\cap D$.

Denote $x_1=I_E\|x\|^{-1}x$ and $y_1=I_E\|x\|^{-1}y$, then $\|x_1\|=I_E,\|y_1\|\leq I_E$ and $A_{x_1y_1}=E$. By Proposition \ref{prop:hard} there exist $u_1,v_1\in S$ with $\|u_1\|=\|v_1\|=I_E$ such that $u_1-v_1=I_E(x_1-y_1)=I_E\|x\|^{-1}(x-y)$ (clearly, $\|u_1-v_1\|\geq I_E\|x-y\|\geq\epsilon I_E$ since $\|x\|\leq 1$) and $\|u_1+v_1\|\geq I_E\|x_1+y_1\|\geq I_E\|x+y\|$.

Similarly, considering $x_2=I_F\|y\|^{-1}x$ and $y_2=I_F\|y\|^{-1}y$ we can have $u_2,v_2\in S$ with $\|u_2\|=\|v_2\|=I_F$ such that $\|u_2-v_2\|\geq\epsilon I_F$ and $\|u_2+v_2\|\geq I_F\|x+y\|$.

Set $u=u_1+u_2$ and $v=v_1+v_2$, then $\|u\|=\|v\|=I_D,\|u-v\|\geq\epsilon I_D$ and $\|u+v\|\geq I_D\|x+y\|$, which completes the proof.\hfill$\Box$

\begin{remark}
If we define a new quantity $$\delta_D^{(3)}(\epsilon)
=\bigwedge\left\{{I}_D-{I}_D\left\|\frac{x+y}{2}\right\|:x, y\in U(1)\textmd{ with } B_{xy}\supset D\textmd{ and
}{I}_D\|x-y\|=\epsilon{I}_D\right\},$$
then by the relation $\delta^{(1)}_D(\epsilon)\geq\delta^{(3)}_D(\epsilon)\geq\delta_D^{(2)}(\epsilon)$ we can also see that $\delta_D(\epsilon)=\delta_D^{(3)}(\epsilon)$ for any
$D\in\widetilde{\cal F}$ with $D\subset G(S)$ and $P(D)>0$ and $\epsilon\in L^{0}_{+}$ such that
$0<\epsilon\leq 2$ on $D$. We should point out that Theorem \ref{thm:expr} automatically hold for every complete complex $RN$ module since a complex $RN$ module $(S,\|\cdot\|)$ can always be viewed as a real $RN$ module with $G(S)=H(S)$.
\end{remark}

For the proof of Theorem \ref{thm:con} we remain to need the following preparations.

First, Lemma \ref{lem:L3.2} means that the estimation of modulus of random convexity given in \cite[Lemma 4.3]{Guo-Zeng} can be improved to the following:

\begin{lemma}\label{lem:L3.3}
Let $(S,\|\cdot\|)$ be a complete $RN$ module over $R$ with base $(\Omega,{\cal F},P)$ such that $P(G(S))>0$, then $\delta_{G(S)}(\epsilon)\leq\frac{\epsilon}{2}I_{G(S)}$ for any $\epsilon\in L^{0}_+$ such that $0<\epsilon\leq 2$ on $G(S)$.
\end{lemma}

We can now in turn obtain Proposition \ref{prop:key} below, which improves \cite[Proposition 4.5]{Guo-Zeng} in that Proposition \ref{prop:key} has removed Condition ($\triangle$) originally imposed on \cite[Proposition 4.5]{Guo-Zeng}.

\begin{proposition}\label{prop:key}
 Let $(S,\|\cdot\|)$ be a complete random
uniformly convex $RN$ module over $K$ with base
$(\Omega,{\cal F},P)$ and $p$ a number such that $1<p<+\infty$. Then for each
number $\epsilon\in (0,2]$ there exists a number
$\delta_p(\epsilon)\in (0,1)$ such that (a) implies (b) for any
$x,y$ in $S$ and any
$D\in\widetilde{\cal F}$ with $D\subset B_{xy} \textmd{ and } P(D)>0$:\\
\indent(a) ${I}_D\|x\|\leq {I}_D,{I}_D\|y\|\leq {I}_D$
and ${I}_D\|x-y\|\geq \epsilon{I}_D$;\\
\indent(b) ${I}_D\left\|\frac{x+y}{2}\right\|^p\leq
{I}_D(1-\delta_p(\epsilon))\frac{\|x\|^p+\|y\|^p}{2}$.\\
Furthermore, (a) can be replaced by\\
\indent(c) ${I}_D\|x-y\|\geq \epsilon{I}_D(\|x\|\vee \|y\|)$.
\end{proposition}

\begin{proof} It is similar to that of \cite[Proposition 4.5]{Guo-Zeng}, except for some key modifications below.

Along the idea of proof of \cite[Proposition 4.5]{Guo-Zeng}, before we consider Part (1) of Case (2) in the proof of \cite[Proposition 4.5]{Guo-Zeng}, we only need to first consider the corresponding problem on $D_1\cap (H(S)\backslash G(S))$ rather than on $D_1$ as in the proof of \cite[Proposition 4.5]{Guo-Zeng}. Therefore, by Corollary \ref{cor:mod} we have $\delta_{H(S)\backslash G(S)}(\gamma)=I_{H(S)\backslash G(S)}$ for any $\gamma\in L^{0}_{+}$ with $0<\gamma\leq 2$ on $H(S)\backslash G(S)$, hence $$I_{D_1\cap (H(S)\backslash G(S))}\varphi(t)=I_{D_1\cap (H(S)\backslash G(S))}\frac{\left(\frac{1-t}{2}\right)^p}{\frac{1+t^p}{2}}.$$

We further notice that the real function $f(s)=\left(\frac{1-s}{2}\right)^p/(\frac{1+s^p}{2})$ is strictly decreasing on $[0,1]$ and take $c_1(p)=f(0)=(1/2)^{p-1}$, then we have $$I_{D_1\cap (H(S)\backslash G(S))}\varphi(t)\leq c_1(p)I_{D_1\cap (H(S)\backslash G(S))}.$$

Next, we replace $D_{11}$ and $D_{12}$ in the original proof of \cite[Proposition 4.5]{Guo-Zeng} with $D_{11}^{\prime}=D_1\cap G(S)\cap [\epsilon^{\prime}<\frac{\epsilon}{2}]$ and $D_{12}^{\prime}=D_1\cap G(S)\cap [\epsilon^{\prime}\geq\frac{\epsilon}{2}]$, respectively. By Lemma \ref{lem:L3.3}, the remaining part of the proof of Proposition \ref{prop:key} is the same as that of \cite[Proposition 4.5]{Guo-Zeng}.
\end{proof}

\vspace{3mm}
We can now prove Theorem \ref{thm:con}.\\
{\noindent\bf Proof of Theorem \ref{thm:con}.} The sufficiency is already known by \cite[Theorem 4.3]{Guo-Zeng}. We only need to prove the necessity. From the proof of \cite[Theorem 4.4]{Guo-Zeng}, if \cite[Proposition 4.5]{Guo-Zeng} is replaced by Proposition \ref{prop:key} then one can see that Theorem \ref{thm:con} always holds.\hfill $\Box$

\begin{remark} It is well known that $L^{p}(\Omega,{\cal F},P)$ is uniformly convex for $1<p<+\infty$ \cite{Clarkson} and that $L^{p}({\cal F},X)$ is uniformly convex iff $X$ is
uniformly convex, where $1<p<+\infty$ and $X$ is an arbitrary Banach space \cite{Day}. Theorem \ref{thm:con} and \cite[Theorem 4.2]{Guo-Zeng} together imply the two well known facts.
\end{remark}
%to the author:   Theorem \ref{thm:HahnB} could be recalled by using Theorem \ref{thm:DD}

\end{document}